\documentclass[12pt, a4paper]{article}
\usepackage{amsfonts}
\usepackage{mathrsfs}
\usepackage{latexsym}
\usepackage{xy}
\usepackage{amsfonts,amsmath,amssymb,amsthm}
\usepackage{color}

\xyoption{all}

\newcommand{\bcen}{\begin{center}}     \newcommand{\ecen}{\end{center}}
\newcommand{\bay}{\begin{array}}      \newcommand{\eay}{\end{array}}
\newcommand{\beq}{\begin{eqnarray*}}      \newcommand{\eeq}{\end{eqnarray*}}

\def\gl{\mathrm{gl.dim}}

\def\hh{\mathrm{hh.dim}}

\def\mod{\mathrm{mod}}

\begin{document}

\newtheorem{theorem}{Theorem}
\newtheorem{proposition}{Proposition}
\newtheorem{lemma}{Lemma}
\newtheorem{corollary}{Corollary}
\newtheorem{remark}{Remark}
\newtheorem{example}{Example}
\newtheorem{definition}{Definition}
\newtheorem*{conjecture}{Conjecture}

\title{\large {\bf Hochschild homology, global dimension, and truncated oriented cycles \footnote{The author
is sponsored by Project 10731070 NSFC.}}}

\author{\large Yang Han}

\date{\footnotesize KLMM, AMSS,
Chinese Academy of Sciences, Beijing 100190, China.\\ E-mail:
hany@iss.ac.cn}

\maketitle

{\it {\footnotesize Dedicated to Professor Claus Michael Ringel on
the occasion of his 65th birthday}}

\begin{abstract} It is shown that a bounded quiver algebra having a 2-truncated
oriented cycle is of infinite Hochschild homology dimension and
global dimension, which generalizes a result of Solotar and
Vigu\'{e}-Poirrier to nonlocal ungraded algebras having a
2-truncated oriented cycle of arbitrary length. Therefore, a bounded
quiver algebra of finite global dimension has no 2-truncated
oriented cycles. Note that the well-known ``no loops conjecture",
which has been proved to be true already, says that a bounded quiver
algebra of finite global dimension has no loops, i.e., truncated
oriented cycles of length 1. Moreover, it is shown that a monomial
algebra having a truncated oriented cycle is of infinite Hochschild
homology dimension and global dimension. Consequently, a monomial
algebra of finite global dimension has no truncated oriented cycles.
\end{abstract}

\medskip


\bigskip

Let $K$ be a field and $Q$ a finite quiver with vertex set $Q_0 :=
\{1, \cdots , n\}$ and arrow set $Q_1$. Denote the source and target
of an arrow $a \in Q_1$ by $s(a)$ and $t(a)$ respectively. Let $R$
be the arrow ideal of the path algebra $KQ$ of the quiver $Q$ and
$I$ an admissible ideal of $KQ$. Denote by $A$ the factor algebra
$KQ/I$ of $KQ$ modulo $I$. Let $e_1, \cdots , e_n$ be the trivial
paths corresponding to the vertices in $Q_0$. Then $S := \oplus
_{i=1}^nKe_i$ is a separable $K$-subalgebra of $A$ and $J := R/I$ is
a two-sided ideal of $A$. Moreover, $A = S \oplus J$. For quivers
and their representations we refer to \cite{ASS}.

Recall that a path $a_1a_2 \cdots a_l$ is called an {\it oriented
cycle} if $l \geq 1, a_i \in Q_1$ and $t(a_i)=s(a_{i+1})$ for all $1
\leq i \leq l$, where $a_{l+1} := a_1$. An oriented cycle $a_1a_2
\cdots a_l$ is said to be an {\it $m$-truncated oriented cycle} of
length $l$ in $A$ if $m \geq 2$ and $a_ia_{i+1} \cdots a_{i+m-1}=0$
but $a_ia_{i+1} \cdots a_{i+m-2} \neq 0$ in $A$ for all $1 \leq i
\leq l$, where we require $a_p = a_q$ if $1 \leq p,q \leq l+m-1$ and
$p \equiv q \; (\mod l)$. So a $2$-truncated oriented cycle in $A$
is just an oriented cycle $a_1a_2 \cdots a_l$ satisfying
$a_ia_{i+1}=0$ in $A$ for all $1 \leq i \leq l$, where again
$a_{l+1} := a_1$. Clearly, if $A$ has an $m$-truncated oriented
cycle of length $l$ then it must have $m$-truncated oriented cycles
of length $il$ for all $i \geq 1$. In particular, if $A$ has a loop
$a \in J^{m-1} \backslash J^m$ then it has an $m$-truncated oriented
cycles of length 1, thus $m$-truncated oriented cycles of arbitrary
length $l \geq 1$.

Recall that the {\it Hochschild homology dimension} of $A$ is $\hh A
:= \inf \{ n \in \mathbb{N}_0 | HH_i(A)=0$ for $i
> n \}$, where $HH_i(A)$ denotes the $i$-th Hochschild homology
group of $A$. (ref. \cite{B,H}).

\begin{theorem} \label{1} If $A$ has $2$-truncated oriented cycles
then $\hh A = \infty = \gl A$. \end{theorem}

\begin{proof} We choose $a_1a_2 \cdots a_l$ to be a $2$-truncated oriented cycle
in $A$ of {\it minimal} length, so that all its {\it rotations}
$(a_1,a_2, \cdots ,a_l), (a_2, \cdots , a_l, a_1), \cdots \cdots ,$
\linebreak $(a_l, a_1, \cdots , a_{l-1})$ are distinct in the set
$Q_1 \times \cdots \times Q_1$ ($l$ copies).

In order to construct the nonzero Hochschild homology classes of
$A$, we observe the {\it $S$-normalized complex} $\bar{C}_S(A) := (A
\otimes_{S^e}(J^{\otimes_Sm}), b)$ of $A$ (ref. \cite[p.134]{C}),
where $b$ is the Hochschild boundary given by $b(x_0, x_1, \cdots ,
x_m) = \sum_{i=0}^{m-1}(-1)^i(x_0, \cdots , x_ix_{i+1}, \cdots ,
x_m) + (-1)^m(x_mx_0, x_1, \cdots , x_{m-1})$.

Now we consider the $(lm-1)$-chain $$\xi := (a_1, \cdots , a_l, a_1,
\cdots , a_l, \cdots \cdots , a_1, \cdots , a_l) \in
\bar{C}_S(A)_{lm-1} = A \otimes_{S^e} J^{\otimes_S (lm-1)}$$ of the
$S$-normalized complex $\bar{C}_S(A)$. Since $a_1a_2 \cdots a_l$ is
a 2-truncated oriented cycle, $b(\xi)=0$, i.e., $\xi$ is a nonzero
$(lm-1)$-cycle of $\bar{C}_S(A)$.

Next we show that $\xi$ is not an $(lm-1)$-boundary, and thus $\xi$
provides a nonzero element $\bar{\xi}$ in the $(lm-1)$-th Hochschild
homology $HH_{lm-1}(A)$ for infinitely many $m$, indeed for at least
all odd $m$. We assume on the contrary that $\xi = b(\sum k(x_0,
x_1, \cdots , x_{lm}))$, where $k \in K$ and $x_0, x_1, \cdots ,
x_{lm}$ are paths in $Q$. Denote by $U$ the $K$-subspace of
$\bar{C}_S(A)_{lm-1}$ generated by all rotations of $\xi$, and by
$V$ the complement space of $U$ in $\bar{C}_S(A)_{lm-1}$.

(1) If the path $x_0$ is nontrivial, i.e., $ x_0 \in J$, then
$b((x_0, x_1, \cdots , x_{lm})) \in J^2 \otimes_{S^e} (J^{\otimes_S
(lm-1)}) + \sum_{0 \leq i \leq lm-2} A \otimes_{S^e} (J^{\otimes_S
i} \otimes_S J^2 \otimes_S J^{\otimes_S (lm-i-2)}) \subseteq V$.

(2) If $x_1, \cdots , x_{lm}$ are not all arrows, i.e., at least one
of them is in $J^2$, then $b((x_0, x_1, \cdots , x_{lm})) \in
\sum_{0 \leq i \leq lm-2} A \otimes_{S^e} (J^{\otimes_S i} \otimes_S
J^2 \otimes_S J^{\otimes_S (lm-i-2)}) \subseteq V$.

(3) If the path $x_1 \cdots x_{lm}$ is not a rotation of the
oriented cycle $\xi$ then $b((x_0, x_1, \cdots , x_{lm})) \in V$.

By the analysis (1)-(3), we may assume that
$$\xi = b(\sum_{i=1}^l k_i (e_{s(a_i)}, a_i, \cdots , a_l, a_1,
\cdots , a_l, \cdots \cdots , a_1, \cdots , a_l, a_1, \cdots ,
a_{i-1})).$$ By the definition of the Hochschild boundary $b$, we
have $$\begin{array}{ccclcl} \xi &
= & & k_1(a_1, \cdots , a_l) & + & (-1)^{lm} k_1(a_l, \cdots , a_{l-1}) \\
& & + & k_2(a_2, \cdots , a_1) & + & (-1)^{lm} k_2(a_1, \cdots , a_l) \\
& & & \cdots \cdots & & \cdots \cdots \\ & & + & k_l(a_l, \cdots ,
a_{l-1}) & + & (-1)^{lm} k_l(a_{l-1}, \cdots , a_{l-2}).
\end{array}$$

Since $a_1a_2 \cdots a_l$ is a $2$-truncated oriented cycle in $A$
of minimal length, all rotations $(a_1, \cdots , a_l), (a_2, \cdots
, a_1), \cdots , (a_l, \cdots , a_{l-1})$ of $\xi$ are $K$-linear
independent in $\bar{C}_S(A)_{lm-1}$. If $l$ is even then we have
$k_{2p} = k_1$ and $k_{2p-1} = -k_1$ for all $1 \leq p \leq
\frac{l}{2}$. Thus $\xi = 0$. It is a contradiction. If $l$ is odd
then we take $m$ to be odd as well and have $k_p = k_1$ for all $1
\leq p \leq l$. Thus again $\xi = 0$. It is also a contradiction.
Hence $\hh A = \infty$.

It follows from \cite[p.110]{Hap} that $\gl A= \infty$. \end{proof}

\begin{remark} From the proof of Theorem~\ref{1} we know that
$\hh A = \infty$ even holds for infinite dimensional quiver algebra
$A$ having $2$-truncated oriented cycles.
\end{remark}

\begin{corollary} \label{2} A bounded quiver algebra of finite global
dimension has no 2-truncated oriented cycles.
\end{corollary}

\begin{remark} The well-known ``no loops conjecture", which has been proved to be true already (ref.
\cite{I,K,L}), says that a bounded quiver algebra of finite global
dimension has no loops. Corollary~\ref{2} implies that a bounded
quiver algebra of finite global dimension has no 2-truncated
oriented cycles as well.
\end{remark}

\begin{remark} Corollary~\ref{2} also can be proved by observing
the minimal projective resolutions of the simple modules
corresponding to the vertices on the 2-truncated oriented cycles.
Indeed, there are always simple direct summands in the syzygies.
\end{remark}

\begin{remark} In \cite{H} the author suggested the conjecture ``Let $A$ be a
bounded quiver algebra over a field $K$. Then $\gl A < \infty$ if
and only if $\hh A = 0$, if and only if $\hh A = \infty$", which is
equivalent to ``Infinite global dimension implies infinite
Hochschild homology dimension for bounded quiver algebras". So far
we have known that the conjecture holds for commutative algebras
\cite{AV}, monomial algebras \cite{H}, quantum complete
intersections of codimension 2 \cite{BE}, graded local algebras,
Koszul algebras and graded cellular algebras over a field of
characteristic zero \cite{BM}. In \cite{SV} the authors proved that
the Hochschild homology dimension of two classes of algebras of
infinite global dimension are infinite. The first class is a
generalization of quantum complete intersections but a
specialization of split extensions of algebras. The second class is
a subclass of finite-dimensional graded local algebras without any
assumption on the underlying field, more precisely, graded local
bounded quiver algebras having a 2-truncated oriented cycle of
length 2. They proved the result for the second class using methods
of differential graded algebra. Theorem~\ref{1} generalizes the
\cite[Theorem II]{SV} to nonlocal ungraded algebras having a
2-truncated oriented cycle of arbitrary length. Nevertheless, our
method is very short and elementary.
\end{remark}

\begin{theorem} \label{3} If $A$ is a monomial algebra having truncated oriented
cycles then $\hh A = \infty = \gl A$. \end{theorem}

\begin{proof} Suppose that $a_1a_2 \cdots
a_l$ is an $m$-truncated oriented cycle in $A$ of length $l$. Let
$A'$ be the bounded quiver algebra $KQ'/I'$ where the quiver $Q'$
has just $l$ vertices $1, 2, \cdots , l$ and $l$ arrows $x_1, x_2,
\cdots , x_l$ such that $t(x_i)=s(x_{i+1})$ for all $1 \leq i \leq
l$ with $x_{l+1} := x_1$, and $I'$ is the admissible ideal of $KQ'$
generated by all paths of length $l$. It is easy to construct a
direct summand of the $S$-normalized complex $\bar{C}_S(A)$
according to the $m$-truncated oriented cycle $a_1a_2 \cdots a_l$
such that it is isomorphic to the $S'$-normalized complex
$\bar{C}_{S'}(A')$ where $S':= \oplus _{i=1}^lKe'_i$ and $e'_1,
\cdots , e'_l$ are the trivial paths corresponding to the vertices
in $Q'_0$. Therefore, $HH_i(A')$ is a direct summand of $HH_i(A)$
for all $i \geq 1$. By \cite[Corollary 1]{H}, we have $\hh A' =
\infty$. Thus $\hh A = \infty$. Furthermore, $\gl A = \infty$.
\end{proof}

\begin{corollary} \label{4} A monomial algebra of finite global dimension
has no truncated oriented cycles. \end{corollary}

\begin{remark} Corollary~\ref{4} also can be proved by observing
the minimal projective resolutions of the simple modules
corresponding to the vertices on the truncated oriented cycle.
\end{remark}

\begin{remark} I don't know whether a bounded quiver algebra of finite global
dimension must have no truncated oriented cycles. Of course a
bounded quiver algebra of infinite global dimension may have no
truncated oriented cycles. For this, it is enough to consider the
algebra $A=KQ/I$ where the quiver $Q$ is given by $Q_0 = \{1,2\}$
and $Q_1 = \{a_1 : 1 \rightarrow 2 , \;\; a_2 : 2 \rightarrow 1\}$
and $I=(a_1a_2a_1)$.
\end{remark}

\end{document}